\documentclass[11pt,reqno]{amsart}

\usepackage{amsmath, amsfonts, amsthm, amssymb, stmaryrd, color, cite}

\usepackage{hyperref}
\newtheorem{theorem}{Theorem}[section]
\newtheorem{lemma}{Lemma}[section]
\newtheorem{proposition}{Proposition}[section]

\theoremstyle{definition}

\theoremstyle{remark}
\newtheorem{remark}{Remark}[section]

\numberwithin{equation}{section}
\allowdisplaybreaks

\def\f{\frac}

\def\hf1{^\f{1}{1-\xi^2}}

\def\be{\begin{equation}}
\def\en{\end{equation}}
\def\bs{\begin{split}}
\def\es{\end{split}}
\def\ba{\begin{align}}
\def\ea{\end{align}}

\allowdisplaybreaks

\author[Z. Qiu]{Zhaoyang Qiu}
\address{School of Mathematics and Statistics, Huazhong University of Science and Technology, Wuhan, 430074, China.}
\email{ZHQMATH@163.com}

\author[H. Wang]{Huaqiao Wang}
\address{College of Mathematics and Statistics, Chongqing University, Chongqing, 401331, China.}
\email{hqwang111@163.com }

\title[ Cahn-Hilliard-Navier-Stokes equation with random initial data]
{Well-posedness for the Cahn-Hilliard-Navier-Stokes equation with random initial data}

\keywords{CH-NS equation, random initial data, global solution, local solution, .}
\subjclass[2010]{35Q35, 76D05, 35A07}

\date{\today}

\begin{document}
\begin{abstract}
We consider the almost sure well-posedness of the Cauchy problem to the Cahn-Hilliard-Navier-Stokes equation with a randomization initial data on a torus $\mathbb{T}^3$. First, we prove the local existence and uniqueness of solution. Furthermore, we prove the global existence and uniqueness of solution and give the relative probability estimate under the asssumption of small initial data.
\end{abstract}

\maketitle
\section{Introduction}

In this paper, we consider the Cahn-Hilliard-Navier-Stokes (CH-NS) equation with a randomization initial data on a torus $\mathbb{T}^3$, which couples the Navier-Stokes equation governing the fluid velocity with a Cahn-Hilliard equation for the relative density of atoms of one of the fluids, reading as follows:
\begin{align}\label{Equ1.1}
\begin{cases}
u_t-\nu_{1}\Delta u+(u\cdot \nabla)u+\nabla \pi-\mu\nabla \phi=0,\\
\phi_t-\nu_{3}\triangle \mu+(u\cdot\nabla)\phi=0,\\
\mu=-\nu_{2}\triangle \phi+\alpha f(\phi),\\
\nabla\cdot u=0,
\end{cases}
\end{align}
where $u, \pi$ and the phase parameter $\phi$ represent the velocity, pressure and the relative concentration of one of the fluids, respectively. The coefficients $\nu_{1}>0, \nu_{3}>0$ denote the kinematic viscosity of the fluid and the mobility constant, respectively. Here $\mu$ stands for the chemical potential of binary mixtures and the two physical parameters $\nu_{2}, \alpha$ describe the interaction between two phases. Especially, the constant $\nu_{2}$ is related to the thickness of the interface.

Define $F(\xi)=\int_{0}^{\xi}f(r)dr$ being typical double-well potential. In physical background, a representative example of $F$ is logarithmic type. Usually, we use a polynomial approximation of the type $F(r) = C_1r^4-C_2r^2$ taking place the type of logarithmic. Therefore, here the most physically relevant case $f(\phi)=c_1\phi^3-c_2\phi$ is considered. Hereafter, for simplicity, we set $c_1=c_2=\nu_{1}=\nu_{2}=\nu_{3}=\alpha=1$ (without loss of generality).

Substantial developments of CN-HS equation have been made in recent years. Cao and Gal \cite{CAO} proved the global existence of solutions for the two-phase fluid system with mixed partial viscosity and mobility, and the global existence and uniqueness of classical solutions for the system without viscosity but with full mobility in 2D. Gal and Grasselli \cite{gal} analyzed the asymptotic behaviour of its variational solution. Abels and Feireisl \cite{Abels} considered the global weak solution to the compressible case. We refer the reader to \cite{FGG,FGK, gal1} for more results. Regarding the stochastic case, Medjo \cite{Mejdo} obtained the existence and uniqueness of variational solution which is strong in probability sense but weak in PDE sense in 2D, to the system forced by a multiplicative noise.

In this paper, we are devoting to proving the global and almost surely local well-posedness of the Cauchy problem to the equation (\ref{Equ1.1}) with a randomization initial data on a torus $\mathbb{T}^3$. Also, we give the probability estimate of the well-poesdness of global solution under the assumption of small initial data. The pioneering works for the randomization initial data problem were developed by Burq and Tzvetkov \cite{burq, burq1} for the nonlinear wave equation, which established the local and global solution respectively. Further development has been achieved, in \cite{du} for MHD equation, in \cite{wang} for quasigeostrophic equation, in \cite{MEN} for nonlinear wave equation of power type, in \cite{zhang,zhang1, cui} for the Navier-Stokes equation. Especially, Zhang and Fang \cite{zhang} also covered the case that the domain is a whole space $\mathbb{R}^N, N\geq 3$.

Note that owing to the coupled construction and higher nonlinearity of the system, it is rather challenging to study it in mathematical. Compared with the global existence result of Navier-Stokes equation \cite{kato}, for closing the estimates we have to restrict parameter $\delta\in (\frac{2}{7},1)$, see Theorem 1.2. Nevertheless, our result still includes a large range of Sobolev space. We remark that both of results can also be extended to the whole space case in the spirit of \cite{zhang}.

Define $\mathcal{H}^\alpha(\mathbb{T}^3):=\left\{f\in H^\alpha({\mathbb{T}^3})~{\rm and }~ \int_{\mathbb{T}^3}fdx=0\right\}$.  For any $\alpha\in \mathbb{R}^+$, $f\in \mathcal{H}^\alpha(\mathbb{T}^3)$ can be written as
\begin{eqnarray}
f(x)=\sum_{k\in \mathbb{Z}_\ast^3}f_ke_k,
\end{eqnarray}
where $e_k = e^{ik\cdot x}$ and $\mathbb{Z}_*^3=\mathbb{Z}^3/\{0,0,0\}$, endowed with the norm
\begin{eqnarray*}
\|f\|^2_{\mathcal{H}^\alpha}\approx\sum_{k\in \mathbb{Z}^3_\ast}|k|^{2\alpha}|f_k|^2.
\end{eqnarray*}
Define the mapping
\begin{eqnarray*}
\omega\longmapsto f^\omega=\sum_{k\in \mathbb{Z}^3_\ast}f_{k}e_kh_k(\omega),
\end{eqnarray*}
from $\Omega$ to $\mathcal{H}^\alpha(\mathbb{T}^3)$, where the sequence $\{h_k(w)\}_{k\in \mathbb{Z}_*^3}$ is independent and identically distribution (iid) Gaussian random variables defined on a fixed probability space $(\Omega, \mathcal{F},\mathbb{P})$. One can check that the mapping is measurable and $f^\omega\in L^2(\Omega;\mathcal{H}^\alpha)$ is $\mathcal{H}^\alpha$-valued random variable. We call $f^\omega$ the randomization of $f$.

Now, we state our main results.

\begin{theorem}[Local existence and uniqueness]\label{the1.1} Assume that $(u_0,\phi_0)\in \mathcal{H}^\alpha\times \mathcal{H}^{\alpha+\frac{3}{2}} $ for any $\alpha\geq \frac{1}{2}$ and $(u_0^\omega, \phi_0^\omega)$ be its randomization,  then, for almost all $\omega\in \Omega$, there exists a $T_\omega\in \left[0,1\right]$ such that equation (\ref{Equ1.1}) with the initial data $(u_0^\omega,\phi_0^\omega)$ has a unique solution satisfying
\begin{eqnarray*}
&&u^\omega-e^{t\triangle}u_0^\omega\in C([0,T_\omega]; \mathcal{H}^{\alpha}(\mathbb{T}^3))\cap L^{32}([0,T_\omega]; L^4(\mathbb{T}^3)),\\
&&\phi^\omega-e^{-t(-\triangle)^2}\phi_0^\omega \in C([0,T_\omega]; \mathcal{H}^{\alpha+\frac{1}{2},4}(\mathbb{T}^3))\cap L^{6}([0,T_\omega]; \mathcal{H}^{\alpha+\frac{3}{2}}(\mathbb{T}^3)).
\end{eqnarray*}
More precisely, there exists an event $\Omega_T$ for each $T\in \left[0,1\right]$ and constants $c_1, c$ such that
\begin{eqnarray*}
\mathbb{P}(\Omega_T)\geq 1-c_1{\rm exp}\left(-\frac{c}{\|(u_0,\phi_0)\|^2_{\mathcal{H}^\alpha \times\mathcal{H}^{\alpha+\frac{3}{2}} }T^{\frac{2}{9}}}\right),
\end{eqnarray*}
then for any $\omega\in \Omega_{T}$, there exists a unique solution $(u^\omega,\phi^\omega)$ on the time interval $[0,T]$.
\end{theorem}

\begin{theorem}[Global existence and uniqueness]\label{the1.2} Assume that initial data $(u_0,\phi_0)$ satisfies $\|(u_0,\phi_0)\|_{{\mathcal{H}^1}\times \mathcal{H}^{2}}\leq \sqrt{-\frac{c}{\ln(1-\varepsilon)}}$  for $\varepsilon\in (0,1)$ and $(u_0^\omega, \phi_0^\omega)$ be its randomization, then there exists an event $\Lambda\subset \Omega$ with $\mathbb{P}(\Lambda)\geq \varepsilon$ such that for all $\omega\in \Lambda$, $\delta\in (\frac{2}{7},1)$, equation (\ref{Equ1.1}) with the initial data $(u_0^\omega, \phi_0^\omega)$ has a unique global solution satisfying
\begin{eqnarray*}
&&u^\omega-e^{t\triangle}u_0^\omega \in C([0,\infty);t^{\frac{1-\delta}{2}}L^{\frac{3}{\delta}}(\mathbb{T}^3))\cap  C([0,\infty);t^{\frac{1}{2}}\mathcal{H}^{1,3}(\mathbb{T}^3)),\\
&&\phi^\omega-e^{-t(-\triangle)^2}\phi_0^\omega\in C([0,\infty);t^{\frac{1-\delta}{3}}\mathcal{H}^{1,{\frac{3}{\delta}}}(\mathbb{T}^3))\cap  C([0,\infty);t^{\frac{1}{3}}\mathcal{H}^{2,3}(\mathbb{T}^3)).
\end{eqnarray*}
\end{theorem}
We remark that $\|f\|_{C([0,\infty); t^pX)}= \|t^pf\|_{C([0,\infty); X)}$ for $p\in (0,1)$ and $X$ is a Banach space.
The rest of paper is to prove Theorems \ref{the1.1}, \ref{the1.2} in Sections \ref{sec2}, \ref{sec3}.

\maketitle
\section{Proof of Theorem 1.1}\label{sec2}
In this section, our main task is to prove Theorem \ref{the1.1}. At the beginning, we use the following lemma to give a stochastic estimate for the randomization initial data.
\begin{lemma}[\!\!\protect{\cite[Lemma 3.1]{burq}}]\label{lem2.1} For all $p\geq 2$ and the sequence $\{c_k\}_{k\in \mathbb{Z}_*^3}\in l^2$, it holds that
\begin{eqnarray*}
\left\|\sum_{k\in \mathbb{Z}_*^3}c_kh_k(\omega)\right\|_{L^p(\Omega)}^2\leq C\sqrt{p}\sum_{k\in \mathbb{Z}_*^3}|c_k|^2,
\end{eqnarray*}
where $C$ is a constant and $\{h_{k}(\omega)\}_{k\in \mathbb{Z}_*^3}$ is the sequence of iid Gaussian random variables.
\end{lemma}

\begin{lemma}\label{lem1.2} For any $a\geq 0$, $m\geq 1, 2\leq p\leq r<\infty$, $2\leq q\leq \infty, \alpha\in \mathbb{R}$, it holds
\begin{eqnarray}
\left\|t^ae^{-t(-\triangle)^m}f^{\omega}\right\|_{L_\omega^rL_t^q\mathcal{H}^{\alpha,p}}\leq C\sqrt{r}\|f\|_{\mathcal{H}^{\alpha}}.
\end{eqnarray}
Moreover, there exist constants $c,c_1$ such that
\begin{eqnarray}\label{2.2}
\mathbb{P}(\mathcal{S}_{\lambda, T})\leq c_1{\rm exp}\left(-\frac{c\lambda^2}{\|f\|^2_{\mathcal{H}^\alpha}}\right),
\end{eqnarray}
where $\mathcal{S}_{\lambda, T}:=\left\{\omega\in \Omega;\; \left\|t^{a}e^{-t(-\triangle)^m}f^\omega\right\|_{L_t^q\mathcal{H}^{\alpha,p}}\geq \lambda\right\}$.
\end{lemma}
\begin{proof}  We divide the proof into two cases: $q> p$ and $q\leq p$ in the spirit of \cite{wang}.
Using the Minkowski inequality and Lemma \ref{lem2.1}, we obtain for $q\leq p$
\begin{align*}
\left\|t^ae^{-t(-\triangle)^m}f^{\omega}\right\|_{L_\omega^rL_t^q\mathcal{H}^{\alpha,p}}&\leq C\left\|\sum_{k\in \mathbb{Z}_*^{3}}|k|^{\alpha}t^ae^{-t|k|^{2m}}f_ke_kh_k(\omega)\right\|_{L_t^qL_x^pL_\omega^r}\nonumber\\
&\leq C\sqrt{r}\left\|\left\||k|^{\alpha}t^ae^{-t|k|^{2m}}f_ke_k\right\|_{l^2}\right\|_{L_t^qL_x^p}\nonumber\\
&\leq C\sqrt{r}\left\||k|^{\alpha}|f_k|\left\|t^ae^{-t|k|^{2m}}\right\|_{L_t^q}\|e_k\|_{L_x^p}\right\|_{l^2}\nonumber\\
&\leq C\sqrt{r}\left\||k|^{\alpha}|f_k|\right\|_{l^2}\nonumber\\
&\leq C\sqrt{r}\|f\|_{\mathcal{H}^\alpha},
\end{align*}
and for $q> p$
\begin{align*}
\left\|t^a e^{-t(-\triangle)^m}f^{\omega}\right\|_{L_\omega^rL_t^q\mathcal{H}^{\alpha,p}}&\leq C\left\|\sum_{k\in \mathbb{Z}_*^{3}}|k|^{\alpha}t^ae^{-t|k|^{2m}}f_ke_kh_k(\omega)\right\|_{L_x^pL_\omega^rL_t^q}\nonumber\\
&\leq C\left\|\sum_{k\in \mathbb{Z}_*^{3}}|k|^{\alpha}\left\|t^ae^{-t|k|^{2m}}\right\|_{L_t^q}|f_ke_kh_k(\omega)|\right\|_{L_x^pL_\omega^r}\nonumber\\
&\leq C\sqrt{r}\left\|\left\||k|^{\alpha}f_ke_k\right\|_{l^2}\right\|_{L_x^p}\nonumber\\
&\leq C\sqrt{r}\left\||k|^{\alpha}f_k\|e_k\|_{L_x^p}\right\|_{l^2}\nonumber\\
&\leq C\sqrt{r}\|f\|_{\mathcal{H}^\alpha},
\end{align*}
where we have used the fact that $\|e_k\|_{L_x^p}\leq C$ uniformly in $k$.

Utilizing the Chebyshev inequality, for every $r\geq p$, there exists constant $c$ such that
\begin{eqnarray}\label{2.5}
\mathbb{P}(\mathcal{S}_{\lambda, T})\leq \left(\frac{c\sqrt{r}\|f\|_{\mathcal{H}^\alpha}}{\lambda}\right)^r.
\end{eqnarray}
Inequality (\ref{2.5}) holds if $\lambda$ satisfies
\begin{eqnarray*}
\lambda\|f\|^{-1}_{\mathcal{H}^\alpha}\leq c\sqrt{p}e.
\end{eqnarray*}
Indeed, we can choose a certain constant $c_1>0$ such that
$$c_1e^{-\frac{c\lambda^2}{\|f\|^2_{\mathcal{H}^\alpha}}}\geq c_1e^{-(c\sqrt{p}e)^2}\geq 1\geq \mathbb{P}(\mathcal{S}_{\lambda, T}).$$
If not, we set
\begin{eqnarray*}
r:=\left(\frac{\lambda}{c\|f\|_{\mathcal{H}^\alpha}e}\right)^2\geq p.
\end{eqnarray*}
This completes the proof.
\end{proof}

Define the group by $\Phi(t)=e^{-t(-\triangle)^m}$ with $m=1,2$, which generated by the equation:
\begin{eqnarray*}
\partial_t v+(-\triangle)^m v=0.
\end{eqnarray*}
Therefore, equation (\ref{Equ1.1}) can be written as the integral form:
\begin{equation}\label{R1.1}
\begin{cases}
u(t)\!=\!e^{t\triangle}u_0^\omega\!-\!\int_0^te^{(t-s)\triangle}\mathrm{P}(u\cdot \nabla u+\triangle\phi\cdot \nabla\phi)ds,\\
\phi(t)\!=\!e^{-t(-\triangle)^2}\phi_0^\omega\!-\!\int_0^te^{-(t-s)(-\triangle)^2}(-u\cdot\nabla \phi+6\phi\cdot|\nabla \phi|^2+3\phi^2\cdot \triangle\phi-\triangle\phi)ds,
\end{cases}
\end{equation}
where $\mathrm{P}={\rm I}+\nabla (-\triangle)^{-1}{\rm div}$ is the Leray-Hopf projector.
\begin{remark} Since $\nabla F(\phi)=f(\phi)\nabla \phi$, and $\mu\nabla \phi=-\Delta \phi\cdot\nabla\phi+\nabla F(\phi)$ in $\eqref{Equ1.1}_1$, the term $\nabla F(\phi)$ could be absorbed into the pressure which was cancelled after applying the Leray-Hopf projector $\mathrm P$.
\end{remark}

Let $\xi=e^{t\triangle}u_0^\omega, \eta=e^{-t(-\triangle)^2}\phi_0^\omega, u=\tilde{u}+\xi, \phi=\tilde{\phi}+\eta$. In order to simplify the notation, we still use $(u,\phi)$ instead of $(\tilde{u}, \tilde{\phi})$, then $(u,\phi)$ satisfies the following integral equation:
\begin{eqnarray}\label{1.7}
\left\{\begin{array}{ll}
\!\!u(t)=-\int_0^te^{(t-s)\triangle}\mathrm{P}b_1(s) ds,\\
\!\!\phi(t)=-\int_0^te^{-(t-s)(-\triangle)^2}b_2(s)ds,
\end{array}\right.
\end{eqnarray}
where
\begin{eqnarray*}
&&b_1=(u+\xi)\cdot \nabla (u+\xi)+\triangle(\phi+\eta)\cdot \nabla(\phi+\eta),\\
&&b_2=-(u+\xi)\cdot\nabla (\phi+\eta)+6(\phi+\eta)\cdot|\nabla (\phi+\eta)|^2+3(\phi+\eta)^2\cdot \triangle(\phi+\eta)-\triangle(\phi+\eta).
\end{eqnarray*}

Define the mapping $\mathcal{M}$ as
$$\mathcal{M}:(u,\phi)\rightarrow \left(-\int_0^te^{(t-s)\triangle}\mathrm{P}b_1(s) ds, -\int_0^te^{-(t-s)(-\triangle)^2}b_2(s)ds\right).$$

We next aim to show that the mapping $\mathcal{M}$ is a contraction. The following two estimates will be applied.

\begin{lemma}[Estimate of $L^p-L^q$]\label{lem2.3}For any $1\leq p\leq q\leq \infty, m\geq 1, a\in \mathbb{R}^{+}$,  it holds
\begin{eqnarray*}
\left\|\partial_x^a e^{-t(-\triangle)^m}f\right\|_{L^q(\mathbb{T}^{n})}\leq Ct^{-\frac{a}{2m}-\frac{n}{2m}(\frac{1}{p}-\frac{1}{q})}\|f\|_{L^p(\mathbb{T}^{n})},
\end{eqnarray*}
where the constant $C$ depends only on $m,n, p,q, a$.
\end{lemma}

\begin{lemma}\label{lem2.4} For any $1\leq r,p,q,p',q'\leq \infty$ with $\frac{1}{r}=\frac{1}{p}+\frac{1}{q}=\frac{1}{p'}+\frac{1}{q'}$ and $a\in \mathbb{R}^+$,  it holds
\begin{eqnarray*}
\|\partial_x^a(fg)\|_{L^r(\mathbb{T}^{n})}\leq C\left(\|f\|_{L^p(\mathbb{T}^{n})}\|\partial_x^ag\|_{L^q(\mathbb{T}^{n})}+\|g\|_{L^{p'}(\mathbb{T}^{n})}\|\partial_x^a f\|_{L^{q'}(\mathbb{T}^{n})}\right),
\end{eqnarray*}
where the constant $C$ depends on $\mathbb{T}^{n}, a, p,q, p',q'$.
\end{lemma}

\begin{proposition}\label{pro2.1} There exists a constant $C$ such that for every $\omega\in E_{\lambda,T}^c, T\in(0,1]$, the mapping $\mathcal{M}$ satisfies
\begin{eqnarray}\label{1.8}
\|\mathcal{M}(u,\phi)\|_{X_T}\leq C T^{\frac{2}{3}}\left(\lambda^2+\lambda^3+\|(u,\phi)\|^2_{X_{T}}+\| \phi\|_{L^\infty_t\mathcal{H}^{\alpha+\frac{1}{2},4}\cap L^6_t\mathcal{H}^{\alpha+\frac{3}{2}} }^3\right),
\end{eqnarray}
and
\begin{align}\label{1.9}
\begin{split}
&\|\mathcal{M}(u_1,\phi_1)-\mathcal{M}(u_2,\phi_2)\|_{X_T}\\
&\leq C T^{\frac{2}{3}}\left(\lambda+\lambda^2+\|(u_1,u_2,\phi_1,\phi_2)\|_{X_{T}}+\| (\phi_1,\phi_2)\|_{L^\infty_t\mathcal{H}^{\alpha+\frac{1}{2},4}\cap L^6_t\mathcal{H}^{\alpha+\frac{3}{2}} }^2\right)\\
&\quad\times\|(u_1-u_2, \phi_1-\phi_2)\|_{X_{T}},
\end{split}
\end{align}
where space
\begin{eqnarray*}
X_T:=L_t^\infty \mathcal{H}^\alpha\cap L_t^{32}L_x^4\times  L^\infty_t\mathcal{H}^{\alpha+\frac{1}{2},4}\cap L^6_t\mathcal{H}^{\alpha+\frac{3}{2}},
\end{eqnarray*}
and set
\begin{eqnarray*}
E_{\lambda,T}:=\left\{\omega\in \Omega; \;\|\xi\|_{L_t^\infty \mathcal{H}^\alpha\cap L_t^{32}L_x^4}+\|\eta\|_{L^\infty_t\mathcal{H}^{\alpha+\frac{1}{2},4}\cap L^6_t\mathcal{H}^{\alpha+\frac{3}{2}}}\geq \lambda\right\}.
\end{eqnarray*}
\end{proposition}
\begin{proof}
\underline{Estimate of ${\rm div}(u\otimes u)$.} Using Lemmas \ref{lem2.3}, \ref{lem2.4} and the H\"{o}lder inequality, we deduce that
\begin{eqnarray}\label{1.10}
&&\left\|\int_{0}^{t}e^{(t-s)\Delta}\mathrm{P}{\rm div}(u\otimes u)ds\right\|_{\mathcal{H}^\alpha}\nonumber\\
&&\leq\int_{0}^{t}\left\|e^{(t-s)\Delta}\mathrm{P}\triangle^{\frac{\alpha}{2}}{\rm div}(u\otimes u)\right\|_{L^2} ds\nonumber\\
&&\leq C\int_{0}^{t}(t-s)^{-\frac{7}{8}}\|u\otimes u\|_{\mathcal{H}^{\alpha, \frac{4}{3}}}ds\\
&&\leq C\|u\|_{L_t^\infty \mathcal{H}^\alpha}\|u\|_{L_t^{32}L_x^4}\left(\int_{0}^{t}(t-s)^{-\frac{28}{31}}ds\right)^\frac{31}{32}\nonumber\\
&&\leq CT^{\frac{3}{32}}\|u\|_{L_t^\infty \mathcal{H}^\alpha}\|u\|_{L_t^{32}L_x^4},\nonumber
\end{eqnarray}
and
\begin{align}\label{1.10s}
\begin{split}
&\left(\int_{0}^{T}\left\|\int_{0}^{t}e^{(t-s)\Delta}\mathrm{P}{\rm div}(u\otimes u)ds\right\|_{L^4}^{32}dt\right)^\frac{1}{32}\\
&\leq\left[\int_{0}^{T}\left(\int_{0}^{t}\left\|e^{(t-s)\Delta}\mathrm{P}{\rm div}(u\otimes u)\right\|_{L^4}ds\right)^{32}dt\right]^\frac{1}{32}\\
&\leq C\left[\int_{0}^{T}\left(\int_{0}^{t}(t-s)^{-\frac{7}{8}}\|u\|_{L^4}^2ds\right)^{32}dt\right]^\frac{1}{32}\\
&\leq CT^{\frac{3}{32}}\|u\|_{L_t^{32}L_x^4}^2.
\end{split}
\end{align}
Note that we use the divergence free condition here. For the terms containing $\xi$, we can use the similar arguments as \eqref{1.10} and \eqref{1.10s} to get
\begin{align}
\begin{split}
& \left\|\int_{0}^{t}e^{(t-s)\Delta}\mathrm{P}{\rm div}(u\otimes \xi)ds\right\|_{\mathcal{H}^\alpha}\le CT^{\frac{3}{32}}(\|\xi\|^2_{{L_t^\infty \mathcal{H}^\alpha}\cap {L_t^{32}L_x^4}}+\|u\|^2_{{L_t^\infty \mathcal{H}^\alpha}\cap {L_t^{32}L_x^4}}),\\
&\left(\int_{0}^{T}\left\|\int_{0}^{t}e^{(t-s)\Delta}\mathrm{P}{\rm div}(u\otimes \xi)ds\right\|_{L^4}^{32}dt\right)^\frac{1}{32}\le CT^{\frac{3}{32}}(\|\xi\|^2_{L_t^{32}L_x^4}+\|u\|^2_{L_t^{32}L_x^4}),
\end{split}
\end{align}
and
\begin{align}
\begin{split}
&\left\|\int_{0}^{t}e^{(t-s)\Delta}\mathrm{P}{\rm div}(\xi\otimes \xi)ds\right\|_{\mathcal{H}^\alpha}\le CT^{\frac{3}{32}}(\|\xi\|^2_{L_t^\infty \mathcal{H}^\alpha}+\|\xi\|^2_{L_t^{32}L_x^4}),\\
&\left(\int_{0}^{T}\left\|\int_{0}^{t}e^{(t-s)\Delta}\mathrm{P}{\rm div}(\xi\otimes \xi)ds\right\|_{L^4}^{32}dt\right)^\frac{1}{32}\le CT^{\frac{3}{32}}\|\xi\|^2_{L_t^{32}L_x^4}.
\end{split}
\end{align}
For convenience, after that, we all omit the estimation of cross-terms such as $\Delta\eta\cdot\nabla\phi, \Delta\phi\cdot\nabla\eta, \Delta\eta\cdot\nabla\eta, (u_1-u_2)\cdot\nabla\xi, \Delta(\phi_1-\phi_2)\cdot\nabla\eta, \Delta\eta\cdot\nabla(\phi_1-\phi_2), \eta^2\Delta(\phi_1-\phi_2),\eta|\nabla(\phi_1-\phi_2)|^2$ like these.

\underline{Estimate of $\triangle \phi\cdot \nabla \phi$.} Also, Lemmas \ref{lem2.3}, \ref{lem2.4} and the H\"{o}lder inequality give
\begin{align}
\begin{split}
&\left\|\int_{0}^{t}e^{(t-s)\Delta}\mathrm{P}(\triangle \phi\cdot \nabla \phi)ds\right\|_{\mathcal{H}^\alpha}\\
&\leq C\int_{0}^{t}\left\|e^{(t-s)\Delta}\triangle^\frac{1}{4}\triangle^{\frac{2\alpha-1}{4}}\mathrm{P}(\triangle \phi\cdot \nabla \phi)\right\|_{L_x^2}ds\\
&\leq C\int_{0}^{t} (t-s)^{-\frac{5}{8}}\|\triangle \phi\cdot \nabla \phi\|_{\mathcal{H}^{\alpha-\frac{1}{2},\frac{4}{3}}}ds\\
&\leq C\int_{0}^{t} (t-s)^{-\frac{5}{8}}\left(\|\triangle\phi\|_{\mathcal{H}^{\alpha-\frac{1}{2}}}\|\nabla \phi\|_{L_x^4}+\|\triangle\phi\|_{L_x^2}\|\nabla \phi\|_{\mathcal{H}^{\alpha-\frac{1}{2},4}}\right)ds\\
&\leq C\|\nabla \phi\|_{L_t^\infty \mathcal{H}^{\alpha-\frac{1}{2},4}}\|\triangle \phi\|_{L_t^6 \mathcal{H}^{\alpha-\frac{1}{2}}}\left(\int_{0}^{t} (t-s)^{-\frac{3}{4}}ds\right)^\frac{5}{6}\\
&\leq CT^{\frac{5}{24}}\|\nabla \phi\|_{L_t^\infty \mathcal{H}^{\alpha-\frac{1}{2},4}}\|\triangle \phi\|_{L_t^6 \mathcal{H}^{\alpha-\frac{1}{2}}},
\end{split}
\end{align}
and
\begin{align}
&\left(\int_{0}^{T}\left\|\int_{0}^{t}e^{(t-s)\triangle} \mathrm{P}(\triangle \phi\cdot \nabla \phi)ds\right\|_{L_x^4}^{32}dt\right)^\frac{1}{32}\nonumber\\
&\leq C\left[\int_{0}^{T}\left(\int_{0}^{t}(t-s)^{-\frac{3}{4}}\|\triangle \phi \cdot \nabla \phi\|_{L_x^{\frac{4}{3}}}ds\right)^{32}dt\right]^\frac{1}{32}\nonumber\\
&\leq C\left[\int_{0}^{T}\left(\int_{0}^{t}(t-s)^{-\frac{3}{4}}\|\triangle \phi\|_{L_x^2} \|\nabla \phi\|_{L_x^{4}}ds\right)^{32}dt\right]^\frac{1}{32}\\
&\leq C\|\nabla \phi\|_{L_t^\infty \mathcal{H}^{\alpha-\frac{1}{2},4}}\left[\int_{0}^{T}\left(\int_{0}^{t}(t-s)^{-\frac{9}{10}}ds\right)^{\frac{80}{3}}\left(\int_{0}^{t}\|\triangle \phi\|_{L_x^2}^{6}ds\right)^\frac{16}{3} dt\right]^\frac{1}{32}\nonumber\\
&\leq CT^\frac{11}{94}\|\nabla \phi\|_{L_t^\infty \mathcal{H}^{\alpha-\frac{1}{2},4}}\|\triangle \phi\|_{L_t^6\mathcal{H}^{s-\frac{1}{2}}}.\nonumber
\end{align}
To close the estimates, we proceed to estimate the phase parameter equation $(\ref{1.7})_2$ in the following.

\noindent \underline{Estimate of $u\cdot \nabla \phi$.} Applying Lemma \ref{lem2.4} and Lemma \ref{lem2.3} with $m=2$, we get
\begin{align}
&\left\|\int_{0}^{t}e^{-(t-s)(-\Delta)^2}\nabla(u\cdot \nabla \phi)ds\right\|_{\mathcal{H}^{\alpha-\frac{1}{2},4}}\nonumber\\
&\leq C\int_{0}^{t}\left\|e^{-(t-s)(-\Delta)^2}\nabla(u\cdot \nabla \phi)\right\|_{\mathcal{H}^{\alpha-\frac{1}{2},4}}ds\nonumber\\
&\leq C\int_{0}^{t}(t-s)^{-\frac{5}{8}}\|u\cdot \nabla \phi\|_{\mathcal{H}^{\alpha-\frac{1}{2},\frac{4}{3}}}ds\\
&\leq C\int_{0}^{t}(t-s)^{-\frac{5}{8}}\left(\|u\|_{\mathcal{H}^\alpha}\|\nabla \phi\|_{L_x^4}+\|u\|_{L_x^4}\|\nabla \phi\|_{\mathcal{H}^{\alpha-\frac{1}{2}}}\right)ds\nonumber\\
&\leq CT^\frac{11}{32}\|u\|_{L_t^\infty \mathcal{H}^\alpha}\left(\|\nabla \phi\|_{L_t^\infty \mathcal{H}^{\alpha-\frac{1}{2},4}}+\|u\|_{L_t^{32}L_x^4}\right),\nonumber
\end{align}
and
\begin{align}\label{1.15}
\begin{split}
&\left(\int_{0}^{T}\left\|\int_{0}^{t}e^{-(t-s)(-\triangle)^2} \triangle(u\cdot \nabla \phi)ds\right\|_{\mathcal{H}^{\alpha-\frac{1}{2}}}^{6}dt\right)^\frac{1}{6}\\
&\leq \left[\int_{0}^{T}\left(\int_{0}^{t}\left\|e^{-(t-s)(-\triangle)^2}\triangle^\frac{3}{4}\triangle^\frac{1}{4} (u\cdot \nabla \phi)\right\|_{\mathcal{H}^{\alpha-\frac{1}{2}}}ds\right)^6dt\right]^\frac{1}{6}\\
&\leq C\left[\int_{0}^{T}\left(\int_{0}^{t}(t-s)^{-\frac{9}{16}}\|\triangle^\frac{1}{4} (u\cdot \nabla \phi)\|_{\mathcal{H}^{\alpha-\frac{1}{2},\frac{4}{3}}}ds \right)^6dt\right]^\frac{1}{6}\\
&\leq C\left[\int_{0}^{T}\left(\int_{0}^{t}(t-s)^{-\frac{9}{16}}\left(\|u\|_{\mathcal{H}^\alpha}\|\nabla \phi\|_{\mathcal{H}^{\alpha-\frac{1}{2},4}}+\|u\|_{L_x^3}\|\nabla \phi\|_{\mathcal{H}^{\alpha,\frac{12}{5}}}\right)ds\right)^6dt\right]^\frac{1}{6}\\
&\leq CT^\frac{1}{2}\|u\|_{L_t^\infty \mathcal{H}^\alpha}\left(\|\nabla \phi\|_{L_t^\infty \mathcal{H}^{\alpha-\frac{1}{2},4}}+\|\triangle \phi\|_{L_t^6\mathcal{H}^{s-\frac{1}{2}}}\right),
\end{split}
\end{align}
where we have used the Sobolev embedding inequality $H^s(\mathbb{T}^3)\hookrightarrow L^3(\mathbb{T}^3)$ for $s\geq \frac{1}{2}$ in (\ref{1.15}).

\noindent\underline{Estimates of $\triangle f(\phi)$.} By Lemmas \ref{lem2.3}, \ref{lem2.4} and the H\"{o}lder inequality again, we conclude that
\begin{align}
\begin{split}
&\left\|\int_{0}^{t}e^{-(t-s)(-\Delta)^2}\nabla(\phi\cdot  |\nabla\phi|^2)ds\right\|_{\mathcal{H}^{\alpha-\frac{1}{2},4}}\\
&\leq C\int_{0}^{t}(t-s)^{-\frac{5}{8}}\|\phi\cdot   |\nabla\phi|^2\|_{\mathcal{H}^{\alpha-\frac{1}{2},\frac{4}{3}}}ds\\
&\leq C\int_{0}^{t}(t-s)^{-\frac{5}{8}}\left(\|\phi\|_{L_x^4}\| |\nabla\phi|^2\|_{\mathcal{H}^{\alpha-\frac{1}{2}}}+\| |\nabla\phi|^2\|_{L_x^2}\|\phi\|_{\mathcal{H}^{\alpha-\frac{1}{2},4}}\right)ds\\
&\leq CT^{\frac{3}{8}}\|\nabla \phi\|^3_{L^\infty_t\mathcal{H}^{\alpha-\frac{1}{2},4}},
\end{split}
\end{align}
and
\begin{align}
\begin{split}
&\left\|\int_{0}^{t}e^{-(t-s)(-\Delta)^2}\nabla(\phi^2\cdot\triangle\phi)ds\right\|_{\mathcal{H}^{\alpha-\frac{1}{2},4}}\\
&\leq C\int_{0}^{t}(t-s)^{-\frac{5}{8}}\|\phi^2\cdot \triangle\phi\|_{\mathcal{H}^{\alpha-\frac{1}{2},\frac{4}{3}}}ds\\
&\leq C\int_{0}^{t}(t-s)^{-\frac{5}{8}}\left(\|\phi^2\|_{L_x^4}\| \triangle\phi\|_{\mathcal{H}^{\alpha-\frac{1}{2}}}+\| \triangle\phi\|_{L_x^2}\|\phi^2\|_{\mathcal{H}^{\alpha-\frac{1}{2},4}}\right)ds\\
&\leq CT^{\frac{5}{24}}\left(\|\nabla \phi\|_{L^\infty_t\mathcal{H}^{\alpha-\frac{1}{2},4}}^3+\| \triangle\phi\|_{L^6_t\mathcal{H}^{\alpha-\frac{1}{2}}}^3\right).
\end{split}
\end{align}
Similarly, the Sobolev embedding inequality $H^{\alpha+\frac{1}{2},4}(\mathbb{T}^3)\hookrightarrow L^{12}(\mathbb{T}^3)$ if $\alpha\geq \frac{1}{2}$, Lemmas \ref{lem2.4}, \ref{lem2.3} and the H\"{o}lder inequality yield that
\begin{align}\label{2.28}
\begin{split}
&\left(\int_{0}^{T}\left\|\int_{0}^{t}e^{-(t-s)(-\triangle)^2} \triangle(\phi^2\cdot \triangle\phi)ds\right\|_{\mathcal{H}^{\alpha-\frac{1}{2}}}^{6}dt\right)^\frac{1}{6}\\
&\leq C\left[\int_{0}^{T}\left(\int_{0}^{t}(t-s)^{-\frac{5}{8}}\|\phi^2\cdot \triangle\phi\|_{\mathcal{H}^{\alpha-\frac{1}{2},\frac{3}{2}}}ds \right)^6dt\right]^\frac{1}{6}\\
&\leq C\left[\int_{0}^{T}\left(\int_{0}^{t}(t-s)^{-\frac{5}{8}}\left(\|\phi^2\|_{L_x^6}\|\triangle \phi\|_{\mathcal{H}^{\alpha-\frac{1}{2}}}+\|\triangle \phi\|_{L_x^2}\|\phi^2 \|_{\mathcal{H}^{\alpha-\frac{1}{2},6}}\right)ds\right)^6dt\right]^\frac{1}{6}\\
&\leq C\left[\int_{0}^{T}\left(\int_{0}^{t}(t-s)^{-\frac{5}{8}}\|\nabla\phi\|^2_{\mathcal{H}^{\alpha-\frac{1}{2},4}}\|\triangle \phi\|_{\mathcal{H}^{\alpha-\frac{1}{2}}}ds\right)^6dt\right]^\frac{1}{6}\\
&\leq CT^{\frac{5}{12}}\left(\|\nabla \phi\|_{L^\infty_t\mathcal{H}^{\alpha-\frac{1}{2},4}}^3+\| \triangle\phi\|_{L^6_t\mathcal{H}^{\alpha-\frac{1}{2}}}^3\right),
\end{split}
\end{align}
and
\begin{align}\label{1.19}
\left(\int_{0}^{T}\left\|\int_{0}^{t}e^{-(t-s)(-\triangle)^2} \triangle(\phi\cdot |\nabla\phi|^2)ds\right\|_{\mathcal{H}^{\alpha-\frac{1}{2}}}^{6}dt\right)^\frac{1}{6}
\leq CT^{\frac{13}{24}}\|\nabla \phi\|_{L^\infty_t\mathcal{H}^{\alpha-\frac{1}{2},4}}^3.
\end{align}
The term $\int_0^te^{-(t-s)(-\triangle)^2}\triangle\phi ds$ is easy to control, here we omit the details.  Finally, combining estimates (\ref{1.10})-(\ref{1.19}), we obtain the inequality (\ref{1.8}) for every $\omega\in E_{\lambda,T}^c$. Here we note that when $\omega\in E_{\lambda,T}^c$, we have $\|\xi\|_{L_t^\infty \mathcal{H}^\alpha\cap L_t^{32}L_x^4}+\|\eta\|_{L^\infty_t\mathcal{H}^{\alpha+\frac{1}{2},4}\cap L^6_t\mathcal{H}^{\alpha+\frac{3}{2}}}\leq \lambda$. Therefore, items with $\lambda^2$ and $\lambda^3$ appear in \eqref{1.8}.

Next, we proceed to prove the inequality (\ref{1.9}). Here we only focus on the higher-order nonlinearity terms $\phi^2\cdot \triangle\phi$ and $\phi\cdot |\nabla\phi|^2$. The rest of terms is similar to the argument of inequality (\ref{1.8}). For the term $\phi^2\cdot \triangle\phi$, Lemmas \ref{lem2.3}, \ref{lem2.4} and the H\"{o}lder inequality yield that
\begin{align*}
&\left\|\int_{0}^{t}e^{-(t-s)(-\Delta)^2}\nabla(\phi_1\cdot  |\nabla\phi_1|^2-\phi_2\cdot  |\nabla\phi_2|^2)ds\right\|_{\mathcal{H}^{\alpha-\frac{1}{2},4}}\\
&=\left\|\int_{0}^{t}e^{-(t-s)(-\Delta)^2}\nabla\left((\phi_1-\phi_2)\cdot  |\nabla\phi_1|^2+\phi_2\cdot  (|\nabla\phi_1|^2-|\nabla\phi_2|^2)\right)ds\right\|_{\mathcal{H}^{\alpha-\frac{1}{2},4}}\\
&\leq C\int_{0}^{t}(t-s)^{-\frac{5}{8}}\left(\|(\phi_1-\phi_2)\cdot   |\nabla\phi_1|^2\|_{\mathcal{H}^{\alpha-\frac{1}{2},\frac{4}{3}}}+\|\phi_2\cdot  (|\nabla\phi_1|^2-|\nabla\phi_2|^2)\|_{\mathcal{H}^{\alpha-\frac{1}{2},\frac{4}{3}}}\right)ds\\
&\leq C\int_{0}^{t}(t-s)^{-\frac{5}{8}}\big(\|\phi_1-\phi_2\|_{L_x^4}\| |\nabla\phi_1|^2\|_{\mathcal{H}^{\alpha-\frac{1}{2}}}+\| |\nabla\phi_1|^2\|_{L_x^2}\|\phi_1-\phi_2\|_{\mathcal{H}^{\alpha-\frac{1}{2},4}}\\
&\quad+\|\phi_2\|_{L_x^4}\||\nabla\phi_1|+|\nabla\phi_2|\|_{\mathcal{H}^{\alpha-\frac{1}{2},4}}\|\nabla(\phi_1-\phi_2)\|_{\mathcal{H}^{\alpha-\frac{1}{2},4}}\\
&\quad+\| \phi_2\|_{\mathcal{H}^{\alpha-\frac{1}{2},4}}\||\nabla\phi_1|+|\nabla\phi_2|\|_{L_x^4}\|\nabla(\phi_1-\phi_2)\|_{L_x^4} \big)ds\\
&\leq CT^{\frac{3}{8}}\|(\nabla \phi_1,\nabla \phi_2)\|^2_{L^\infty_t\mathcal{H}^{\alpha-\frac{1}{2},4}}\|\nabla(\phi_1-\phi_2)\|_{L_t^\infty\mathcal{H}^{\alpha-\frac{1}{2},4}},
\end{align*}
and
\begin{align*}
&\left\|\int_{0}^{t}e^{-(t-s)(-\Delta)^2}\nabla(\phi_1^2\cdot\triangle\phi_1-\phi_2^2\cdot\triangle\phi_2)ds\right\|_{\mathcal{H}^{\alpha-\frac{1}{2},4}}\\
&\leq C\int_{0}^{t}(t-s)^{-\frac{5}{8}}\left(\|(\phi_1^2-\phi_2^2)\cdot \triangle\phi_1\|_{\mathcal{H}^{\alpha-\frac{1}{2},\frac{4}{3}}}
+\|\phi_2^2\cdot(\triangle\phi_1-\triangle\phi_2)\|_{\mathcal{H}^{\alpha-\frac{1}{2},\frac{4}{3}}}\right)ds\\
&\leq C\int_{0}^{t}(t-s)^{-\frac{5}{8}}\big(\|\phi_1^2-\phi_2^2\|_{L_x^4}\| \triangle\phi_1\|_{\mathcal{H}^{\alpha-\frac{1}{2}}}+\| \triangle\phi_1\|_{L_x^2}\|\phi_1^2-\phi_2^2\|_{\mathcal{H}^{\alpha-\frac{1}{2},4}}\\
&\quad+\|\phi_2^2\|_{L_x^4}\| \triangle\phi_1-\triangle\phi_2\|_{\mathcal{H}^{\alpha-\frac{1}{2}}}+\| \triangle\phi_1-\triangle\phi_2\|_{L_x^2}\|\phi_2^2\|_{\mathcal{H}^{\alpha-\frac{1}{2},4}}\big)ds\\
&\leq CT^{\frac{5}{24}}\bigg(\|\nabla \phi_1-\nabla \phi_2\|_{L^\infty_t\mathcal{H}^{\alpha-\frac{1}{2},4}}\|(\nabla \phi_1, \nabla \phi_2)\|_{L^\infty_t\mathcal{H}^{\alpha-\frac{1}{2},4}}\| \triangle\phi\|_{L^6_t\mathcal{H}^{\alpha-\frac{1}{2}}}\\
&\quad+\| \triangle\phi_1-\triangle\phi_2\|_{L^6_t\mathcal{H}^{\alpha-\frac{1}{2}}}\|\nabla\phi_2\|^2_{L_t^\infty\mathcal{H}^{\alpha-\frac{1}{2},4}}\bigg).
\end{align*}
Similarly, we have
\begin{align*}
&\left(\int_{0}^{T}\left\|\int_{0}^{t}e^{-(t-s)(-\triangle)^2} \triangle(\phi_1^2\cdot \triangle\phi_1-\phi_2^2\cdot \triangle\phi_2)ds\right\|_{\mathcal{H}^{\alpha-\frac{1}{2}}}^{6}dt\right)^\frac{1}{6}\\
&\leq CT^{\frac{5}{12}}\bigg(\|\nabla \phi_1-\nabla \phi_2\|_{L^\infty_t\mathcal{H}^{\alpha-\frac{1}{2},4}}\|(\nabla \phi_1, \nabla \phi_2)\|_{L^\infty_t\mathcal{H}^{\alpha-\frac{1}{2},4}}\| \triangle\phi\|_{L^6_t\mathcal{H}^{\alpha-\frac{1}{2}}}\\
&\quad+\| \triangle\phi_1-\triangle\phi_2\|_{L^6_t\mathcal{H}^{\alpha-\frac{1}{2}}}\|\nabla\phi_2\|^2_{L_t^\infty\mathcal{H}^{\alpha-\frac{1}{2},4}}\bigg),
\end{align*}
as well as
\begin{align*}
&\left(\int_{0}^{T}\left\|\int_{0}^{t}e^{-(t-s)(-\triangle)^2} \triangle(\phi_1\cdot |\nabla\phi_1|^2-\phi_2\cdot |\nabla\phi_2|^2)ds\right\|_{\mathcal{H}^{\alpha-\frac{1}{2}}}^{6}dt\right)^\frac{1}{6}\\
&\leq CT^{\frac{13}{24}}\|(\nabla \phi_1,\nabla \phi_2)\|^2_{L^\infty_t\mathcal{H}^{\alpha-\frac{1}{2},4}}\|\nabla(\phi_1-\phi_2)\|_{L_t^\infty\mathcal{H}^{\alpha-\frac{1}{2},4}}.
\end{align*}
Then, inequality (\ref{1.9}) follows from the above. Thus, we complete the proof of Proposition \ref{pro2.1}. Note that if $\omega\in E_{\lambda,T}^c$, $\|\xi\|_{L_t^\infty \mathcal{H}^\alpha\cap L_t^{32}L_x^4}+\|\eta\|_{L^\infty_t\mathcal{H}^{\alpha+\frac{1}{2},4}\cap L^6_t\mathcal{H}^{\alpha+\frac{3}{2}}}\leq \lambda$. Hence, items with $\lambda$ and $\lambda^2$ appear in \eqref{1.9}.
\end{proof}
Now, we begin to prove Theorem \ref{the1.1} in the following.

\noindent\underline{Proof of Theorem 1.1}. For any fixed $T\in (0,1]$, from Proposition \ref{pro2.1}, we choose $T^{\frac{2}{3}}\lambda^6=\varepsilon^6\ll 1$ for $\varepsilon>0$ such that
\begin{eqnarray*}
&&C T^{\frac{2}{3}}\left(\lambda^2+\lambda^3+(4C\lambda^3)^2+(4C\lambda^3)^3\right)\leq 4C\lambda^3,\\
&&C T^{\frac{2}{3}}\left(\lambda+\lambda^2+4C\lambda^3+(4C\lambda^3)^2\right)\leq \frac{1}{2}.
\end{eqnarray*}
In the spirit of N. Burq and N. Tzvetkov \cite{burq,burq1},  we define the set
\begin{eqnarray}\label{1.11}
\Omega_{T}=E^c_{\lambda=\varepsilon T^{-\frac{1}{9}}, T}~ {\rm and}~\widetilde{\Omega}=\bigcup_{j\in \mathbb{N}^+} \Omega_{\frac{1}{j}},
\end{eqnarray}
where $E^c_{\lambda, T}$ is the complement of set $E_{\lambda, T}$. We conclude from Lemma \ref{lem1.2} (\eqref{2.2}) that
\begin{eqnarray}\label{2.19}
\mathbb{P}(\Omega_T)=1-\mathbb{P}(\Omega_T^c)\geq 1-c_1{\rm exp}\left(-\frac{c}{\|(u_0,\phi_0)\|^2_{\mathcal{H}^\alpha \times \mathcal{H}^{\alpha+\frac{3}{2}}}T^{\frac{2}{9}}}\right).
\end{eqnarray}
and $\mathbb{P}(\widetilde{\Omega})=1$. Therefore, for each $\omega\in \widetilde{\Omega}$, there exists $j\in \mathbb{N}^+$ such that $w\in E^c_{\lambda=\varepsilon j^{\frac{1}{9}}, j^{-1}}$. By Proposition \ref{pro2.1}, we obtain for the fixed $\omega$, the mapping $\mathcal{M}$ is a contraction in ball $B(0, 4C\lambda^3)$ in $X_{1/j}$. The Banach fixed point theorem implies that for almost surely $\omega\in \Omega$, there exists a $T_\omega$ such that $(u^\omega,\phi^\omega)$ is a unique solution in $X_{T_\omega}$. Thus, we get the first part of Theorem \ref{the1.1}. Following the similar argument and (\ref{2.19}), we are able to finish the proof of Theorem \ref{the1.1}.

\maketitle
\section{Proof of Theorem 1.2}\label{sec3}
In this section, we are going to prove Theorem \ref{the1.2}. First, we give some useful lemmas and proposition.
\begin{lemma}\label{lem3.1} For $m\geq 1$, $\gamma,l\geq 0$, $\alpha,\beta>0$ with $\alpha+\beta\leq n$, there exists a constant $C$ such that
\begin{eqnarray*}
\left\|\int_{0}^{t}e^{-(t-s)(-\triangle)^m}\partial_x^\gamma (fg)ds\right\|_{L_x^\frac{n}{l}}\leq C\int_{0}^{t}(t-s)^{-\frac{\alpha+\beta-l+\gamma}{2m}}\|f\|_{L_x^\frac{n}{\alpha}}\|g\|_{L_x^\frac{n}{\beta}}ds.
\end{eqnarray*}
\end{lemma}
\begin{proof} The result can be easily obtained by Lemma \ref{lem2.3} and the H\"{o}lder inequality.
\end{proof}

\begin{lemma}\label{lem3.2} For $p>0, q>0$, there exists a constant $C(p,q)$ such that
\begin{eqnarray*}
\int_{0}^{t}(t-s)^{p-1}s^{q-1}ds=C(p,q)t^{p+q-1}.
\end{eqnarray*}
\end{lemma}
\begin{proof} Since
\begin{align*}
\int_{0}^{t}(t-s)^{p-1}s^{q-1}ds&=\int_{0}^{t}t^{p-1}\left(1-\frac{s}{t}\right)^{p-1}t^{q-1}\left(\frac{s}{t}\right)^{q-1}ds\\
&=\int_{0}^{t}t^{p+q-1}\left(1-\frac{s}{t}\right)^{p-1}\left(\frac{s}{t}\right)^{q-1}d\left(\frac{s}{t}\right),
\end{align*}
it follows from variable substitutions that
\begin{align*}
&\int_{0}^{t}t^{p+q-1}\left(1-\frac{s}{t}\right)^{p-1}\left(\frac{s}{t}\right)^{q-1}d\left(\frac{s}{t}\right)\\
&=t^{p+q-1}\int_{0}^{1}(1-x)^{p-1}x^{q-1}dx\\
&=C(p,q)t^{p+q-1}.
\end{align*}
Here we used the fact that Beta function $B(p,q)$ converges for any $p,q >0$.
\end{proof}

\begin{proposition}\label{pro3.1} There exists a constant $C$ which is independent of $T$ such that for every $\omega\in \widetilde{E}_{\lambda,T}^c$, $\delta\in (\frac{2}{7}, 1)$, the mapping $\mathcal{M}$ satisfies
\begin{eqnarray}\label{3.2}
\|\mathcal{M}(u,\phi)\|_{\widetilde{X}_T}\leq C\left(\lambda^2+\lambda^3+\|(u,\phi)\|^2_{\widetilde{X}_{T}}+\| \phi\|_{L^\infty_t (t^{\frac{1-\delta}{3}}\mathcal{H}^{1,\frac{3}{\delta}})\cap L^\infty_t (t^{\frac{1}{3}}\mathcal{H}^{2,3} )}^3\right),
\end{eqnarray}
and
\begin{align}\label{3.3}
\begin{split}
&\|\mathcal{M}(u_1,\phi_1)-\mathcal{M}(u_2,\phi_2)\|_{\widetilde{X}_T}\\
&\leq C\left(\lambda+\lambda^2+\|(u_1,u_2,\phi_1,\phi_2)\|_{\widetilde{X}_{T}}+\| (\phi_1,\phi_2)\|_{L^\infty_t (t^{\frac{1-\delta}{3}}\mathcal{H}^{1,\frac{3}{\delta}})\cap L^\infty_t (t^{\frac{1}{3}}\mathcal{H}^{2,3} )}^2\right)\\
&\quad\times\|(u_1-u_2, \phi_1-\phi_2)\|_{\widetilde{X}_{T}},
\end{split}
\end{align}
where space
\begin{eqnarray*}
\widetilde{X}_T:=L_t^\infty (t^{\frac{1-\delta}{2}}L_x^{\frac{3}{\delta}})\cap L_t^\infty (t^{\frac{1}{2}}\mathcal{H}^{1})\times L^\infty_t (t^{\frac{1-\delta}{3}}\mathcal{H}^{1,\frac{3}{\delta}})\cap L^\infty_t (t^{\frac{1}{3}}\mathcal{H}^{2,3} ),
\end{eqnarray*}
and set
\begin{eqnarray*}
\widetilde{E}_{\lambda,T}:=\left\{\omega\in \Omega; \;\|\xi\|_{L_t^\infty (t^{\frac{1-\delta}{2}}L_x^{\frac{3}{\delta}})\cap L_t^\infty (t^{\frac{1}{2}}\mathcal{H}^{1})}+\|\eta\|_{L^\infty_t (t^{\frac{1-\delta}{3}}\mathcal{H}^{1,\frac{3}{\delta}})\cap L^\infty_t (t^{\frac{1}{3}}\mathcal{H}^{2,3} )}\geq \lambda\right\}.
\end{eqnarray*}
\end{proposition}

\begin{proof} As in Proposition \ref{pro2.1}, in order to close the estimates, we estimate all terms one by one. Taking $\alpha=\delta$, $\beta=1$ in Lemma \ref{lem3.1} and using Lemma \ref{lem3.2}, we get
\begin{align}\label{3.4}
\begin{split}
&\left\|\int_{0}^{t}e^{(t-s)\triangle} (u\cdot \nabla u) ds\right\|_{L_t^\infty(t^{\frac{1-\delta}{2}}L_x^{\frac{3}{\delta}})}\\
&\leq C\sup_{t\in [0,T]}t^{\frac{1-\delta}{2}}\int_{0}^{t}(t-s)^{-\frac{1}{2}}\|u\|_{L_x^\frac{3}{\delta}}\|\nabla u\|_{L_x^3}ds\\
&\leq C\sup_{t\in [0,T]}t^{\frac{1-\delta}{2}}\int_{0}^{t}(t-s)^{-\frac{1}{2}}s^{-1+\frac{\delta}{2}}\left(s^{\frac{1-\delta}{2}}\|u\|_{L_x^\frac{3}{\delta}}\right)
\left(s^{\frac{1}{2}}\|\nabla u\|_{L_x^3}\right)ds\\
&\leq C\left\|t^{\frac{1-\delta}{2}}\|u\|_{L_x^\frac{3}{\delta}}\right\|_{L_t^\infty}\left\|t^{\frac{1}{2}}\|\nabla u\|_{L_x^3}\right\|_{L_t^\infty},
\end{split}
\end{align}
and
\begin{align}
\begin{split}
&\left\|\int_{0}^{t}e^{(t-s)\triangle} \nabla(u\cdot \nabla u) ds\right\|_{L_t^\infty(t^{\frac{1}{2}}L_x^{3})}\\
&\leq C\sup_{t\in [0,T]}t^{\frac{1}{2}}\int_{0}^{t}(t-s)^{-\frac{\delta+1}{2}}\|u\|_{L_x^\frac{3}{\delta}}\|\nabla u\|_{L_x^3}ds\\
&\leq C\sup_{t\in [0,T]}t^{\frac{1}{2}}\int_{0}^{t}(t-s)^{-\frac{\delta+1}{2}}s^{-1+\frac{\delta}{2}}\left(s^{\frac{1-\delta}{2}}\|u\|_{L_x^\frac{3}{\delta}}\right)
\left(s^{\frac{1}{2}}\|\nabla u\|_{L_x^3}\right)ds\\
&\leq C\left\|t^{\frac{1-\delta}{2}}\|u\|_{L_x^\frac{3}{\delta}}\right\|_{L_t^\infty}\left\|t^{\frac{1}{2}}\|\nabla u\|_{L_x^3}\right\|_{L_t^\infty}.
\end{split}
\end{align}
Taking $\alpha=\frac{2+2\delta}{3}$, $\beta=1$ in Lemma \ref{lem3.1} and applying Lemma \ref{lem3.2}, one has
\begin{align}
&\left\|\int_{0}^{t}e^{(t-s)\triangle} (\triangle \phi\cdot \nabla \phi) ds\right\|_{L_t^\infty(t^{\frac{1-\delta}{2}}L_x^{\frac{3}{\delta}})}\nonumber\\
&\leq C\sup_{t\in [0,T]}t^{\frac{1-\delta}{2}}\int_{0}^{t}(t-s)^{-\frac{5-\delta}{6}}\|\nabla \phi\|_{L_x^\frac{3}{\alpha}}\|\triangle\phi\|_{L_x^3}ds\nonumber\\
&\leq C\sup_{t\in [0,T]}t^{\frac{1-\delta}{2}}\int_{0}^{t}(t-s)^{-\frac{5-\delta}{6}}s^{-\frac{2-\delta}{3}}\left(s^{\frac{1-\delta}{3}}\|\nabla \phi\|_{L_x^\frac{3}{\delta}}\right)
\left(s^{\frac{1}{3}}\|\triangle u\|_{L_x^3}\right)ds\\
&\leq C\sup_{t\in [0,T]}t^{\frac{1-\delta}{2}}t^{\frac{\delta-1}{2}}\left\|\nabla \phi\|_{L_x^\frac{3}{\delta}}\right\|_{L_t^\infty}\left\|t^{\frac{1}{3}}\|\triangle \phi\|_{L_x^3}\right\|_{L_t^\infty}\nonumber\\
&\leq C\left\|t^{\frac{1-\delta}{3}}\|\nabla \phi\|_{L_x^\frac{3}{\delta}}\right\|_{L_t^\infty}\left\|t^{\frac{1}{3}}\|\triangle \phi\|_{L_x^3}\right\|_{L_t^\infty}.\nonumber
\end{align}

Also, we have by choosing $\alpha=\frac{2+2\delta}{3}$, $\beta=1$ in Lemma \ref{lem3.1} and using Lemma \ref{lem3.2}
\begin{align}
\left\|\int_{0}^{t}e^{(t-s)\triangle} \nabla(\triangle \phi\cdot \nabla \phi) ds\right\|_{L_t^\infty(t^{\frac{1}{3}}L_x^{3})}
\leq C\left\|t^{\frac{1-\delta}{3}}\|\nabla \phi\|_{L_x^\frac{3}{\delta}}\right\|_{L_t^\infty}\left\|t^{\frac{1}{3}}\|\triangle\phi\|_{L_x^3}\right\|_{L_t^\infty}.
\end{align}

Next, we focus on the the phase parameter equation. Taking $\alpha=1,\beta=\delta$ and $\alpha=\frac{4\delta}{3}, \beta=1$ in Lemma \ref{lem3.1}, respectively, we obtain from Lemma \ref{lem3.2}
\begin{align}
\begin{split}
&\left\|\int_{0}^{t}e^{-(t-s)(-\triangle)^2} \nabla(u\cdot \nabla \phi) ds\right\|_{L_t^\infty(t^{\frac{1-\delta}{3}}L_x^{\frac{3}{\delta}})}\\
&\leq C\sup_{t\in [0,T]}t^{\frac{1-\delta}{3}}\int_{0}^{t}(t-s)^{-\frac{1}{2}}\|u\|_{L_x^3}\|\nabla\phi\|_{L_x^\frac{3}{\delta}}ds\\
&\leq C\sup_{t\in [0,T]}t^{\frac{1-\delta}{3}}\int_{0}^{t}(t-s)^{-\frac{1}{2}}s^{-\frac{5-2\delta}{6}}\left(s^{\frac{1}{2}}\|\nabla u\|_{L_x^3}\right)\left(s^{\frac{1-\delta}{3}}\|\nabla\phi\|_{L_x^\frac{3}{\delta}}\right)ds\\
&\leq C\left\|t^{\frac{1}{2}}\|\nabla u\|_{L_x^3}\right\|_{L_t^\infty}\left\|t^{\frac{1-\delta}{3}}\|\nabla \phi\|_{L_x^\frac{3}{\delta}}\right\|_{L_t^\infty},
\end{split}
\end{align}
and
\begin{align}
\begin{split}
&\left\|\int_{0}^{t}e^{-(t-s)(-\triangle)^2} \triangle(u\cdot \nabla \phi) ds\right\|_{L_t^\infty(t^{\frac{1}{3}}L_x^{3})}\\
&\leq C\sup_{t\in [0,T]}t^{\frac{1}{3}}\int_{0}^{t}(t-s)^{-\frac{2\delta+3}{6}}s^{-\frac{5-2\delta}{6}}\left(s^{\frac{1}{2}}\|\nabla u\|_{L_x^3}\right)
\left(s^{\frac{1-\delta}{3}}\|\nabla \phi\|_{L_x^\frac{9}{4\delta}}\right)ds\\
&\leq C\left\|t^{\frac{1}{2}}\|\nabla u\|_{L_x^3}\right\|_{L_t^\infty}\left\|t^{\frac{1-\delta}{3}}\|\nabla \phi\|_{L_x^\frac{3}{\delta}}\right\|_{L_t^\infty}.
\end{split}
\end{align}
We proceed to estimate the higher nonlinearity terms $\triangle f(\phi)$. Taking $\alpha=\frac{5\delta+1}{3}, \beta=\delta$ in Lemma \ref{lem3.1} and using Lemma \ref{lem3.2} lead to
\begin{align}
\begin{split}
&\left\|\int_{0}^{t}e^{-(t-s)(-\Delta)^2}\nabla(\phi\cdot  |\nabla\phi|^2)ds\right\|_{L_t^\infty(t^{\frac{1-\delta}{3}}L_x^{\frac{3}{\delta}})}\\
&\leq C\sup_{t\in [0,T]}t^{\frac{1-\delta}{3}}\int_{0}^{t}(t-s)^{-\frac{\alpha+2\beta-\delta+1}{4}}\|\phi\|_{\frac{3}{\alpha}}\||\nabla \phi|^2\|_{\frac{3}{2\beta}}ds\\
&\leq C\sup_{t\in [0,T]}t^{\frac{1-\delta}{3}}\int_{0}^{t}(t-s)^{-\frac{2\delta+1}{3}}s^{-(1-\delta)}ds \left\|t^{\frac{1-\delta}{3}}\|\nabla \phi\|_{L_x^\frac{3}{\delta}}\right\|_{L_t^\infty}^3\\
%&\leq C\sup_{t\in [0,T]}t^{\frac{1-\delta}{3}}t^{\frac{\delta-1}{3}}\left\|t^{\frac{1-\delta}{3}}\|\nabla \phi\|_{L_x^\frac{3}{\delta}}\right\|_{L_t^\infty}^3\nonumber\\
&\leq C\left\|t^{\frac{1-\delta}{3}}\|\nabla \phi\|_{L_x^\frac{3}{\delta}}\right\|_{L_t^\infty}^3,
\end{split}
\end{align}
and taking $\alpha=\frac{7\delta-2}{6}$ for $\delta>\frac{2}{7}$, $\beta=1+\delta$ in Lemma \ref{lem3.1} and using Lemma \ref{lem3.2} yield that
\begin{align}
&\left\|\int_{0}^{t}e^{-(t-s)(-\Delta)^2}\nabla(\phi^2\cdot\triangle\phi)ds\right\|_{L_t^\infty(t^{\frac{1-\delta}{3}}L_x^{\frac{3}{\delta}})}\nonumber\\
&\leq C\sup_{t\in [0,T]}t^{\frac{1-\delta}{3}}\int_{0}^{t}(t-s)^{-\frac{\delta+1}{3}}s^{-\frac{3-2\delta}{3}}\left(s^{\frac{1-\delta}{3}}\|\nabla \phi\|_{L_x^\frac{3}{\delta}}\right)^2\left(s^\frac{1}{3}\|\triangle \phi\|_{L^3_x}\right)ds\\
&\leq C \left\|t^{\frac{1-\delta}{3}}\|\nabla \phi\|_{L_x^\frac{3}{\delta}}\right\|_{L_t^\infty}^2\left\|t^{\frac{1}{3}}\|\triangle\phi\|_{L_x^3}\right\|_{L_t^\infty}.\nonumber
\end{align}
Again, choosing $\alpha=\frac{4\delta-1}{3}$ for $\delta>\frac{1}{4}$, $\beta=1$ and $\alpha=\frac{2\delta+1}{3}$, $\beta=\delta$ in Lemma \ref{lem3.1} and employing Lemma \ref{lem3.2}, we have
\begin{align}
\begin{split}
&\left\|\int_{0}^{t}e^{-(t-s)(-\triangle)^2} \triangle(\phi^2\cdot \triangle\phi)ds\right\|_{L_t^\infty(t^{\frac{1}{3}}L_x^{3})}\\
&\leq C\sup_{t\in [0,T]}t^{\frac{1}{3}}\int_{0}^{t}(t-s)^{-\frac{2\delta+1}{3}}\|\phi\|_{L_x^\frac{3}{\alpha}}^2\|\triangle \phi\|_{L_x^3}ds\\
&\leq C\sup_{t\in [0,T]}t^{\frac{1}{3}}\int_{0}^{t}(t-s)^{-\frac{2\delta+1}{3}}s^{-\frac{3-2\delta}{3}}\left(s^{\frac{1-\delta}{3}}\|\nabla \phi\|_{L_x^\frac{3}{\delta}}\right)^2\left(s^{\frac{1}{3}}\|\triangle \phi\|_{L^3_x}\right)ds\\
&\leq C\left\|t^{\frac{1-\delta}{3}}\|\nabla \phi\|_{L_x^\frac{3}{\delta}}\right\|_{L_t^\infty}^2\left\|t^{\frac{1}{3}}\|\triangle\phi\|_{L_x^3}\right\|_{L_t^\infty},
\end{split}
\end{align}
and
\begin{align}\label{3.13}
\begin{split}
&\left\|\int_{0}^{t}e^{-(t-s)(-\Delta)^2}\triangle(\phi\cdot  |\nabla\phi|^2)ds\right\|_{L_t^\infty(t^{\frac{1}{3}}L_x^{3})}\\
&\leq C\sup_{t\in [0,T]}t^{\frac{1}{3}}\int_{0}^{t}(t-s)^{-\frac{2\delta+1}{3}}\|\phi\|_{L_x^\frac{3}{\alpha}}\left\||\nabla \phi|^2\right\|_{L_x^\frac{3}{2\beta}}ds\\
&\leq C\sup_{t\in [0,T]}t^{\frac{1}{3}}\int_{0}^{t}(t-s)^{-\frac{2\delta+1}{3}}s^{-\frac{2(1-\delta)+1}{3}}\left(t^{\frac{1-\delta}{3}}\|\nabla \phi\|_{L_x^\frac{3}{\delta}}\right)^2\left(s^{\frac{1}{3}}\|\triangle \phi\|_{L_x^3}\right)ds\\
&\leq C\left\|t^{\frac{1-\delta}{3}}\|\nabla \phi\|_{L_x^\frac{3}{\delta}}\right\|_{L_t^\infty}^2\left\|t^{\frac{1}{3}}\|\triangle\phi\|_{L_x^3}\right\|_{L_t^\infty}.
\end{split}
\end{align}
The desired result \eqref{3.2} follows from estimates \eqref{3.4}-\eqref{3.13}. Similarly, we can get \eqref{3.3}.
\end{proof}

Now, we are in the position to prove Theorem \ref{the1.2}.

\noindent\underline{Proof of Theorem \ref{the1.2}}. From Proposition \ref{pro3.1}, we choose $\lambda\in (0,1)$ such that
\begin{eqnarray*}
&&C\left(\lambda^2+\lambda^3+(2C\lambda^2)^2+(2C\lambda^2)^3\right)\leq 2C\lambda^2,\\
&&C\left(\lambda+\lambda^2+2C\lambda^2+(2C\lambda^2)^2\right)\leq \frac{1}{2}.
\end{eqnarray*}
By Proposition \ref{pro3.1}, we infer that the mapping $\mathcal{M}$ is a contraction in a ball $B(0,2C\lambda^2)$ in $\widetilde{X}_{T}$. Let
\begin{eqnarray}
\Omega_T=\widetilde{E}_{\lambda,T}^c.%,~ {\rm and}~~\Gamma=\bigcup_{i\geq 1}\Omega_{2^i}.
\end{eqnarray}
For any $T>0$, there exists $i>0$ such that $2^{i-1}\leq T<2^i$, $\Omega_{2^i}\subset\Omega_{T}\subset \Omega_{2^{i-1}}$. In addition, by Lemma \ref{lem1.2} (\eqref{2.2})
\begin{eqnarray}\label{3.14}
\mathbb{P}(\Omega_{2^{i-1}})\geq 1-\mathbb{P}(\widetilde{E}_{\lambda,2^{i-1}})\geq 1-c_1{\rm exp}\left(-\frac{c\lambda^2}{\|(u_0, \phi_0)\|^2_{\mathcal{H}^{1}\times \mathcal{H}^2}}\right).
\end{eqnarray}
For any $\varepsilon\in (0,1)$, let
\begin{align*}
1-c_1{\rm exp}\left(-\frac{c\lambda^2}{\|(u_0, \phi_0)\|^2_{\mathcal{H}^{1}\times \mathcal{H}^2}}\right)\geq \varepsilon.
\end{align*}
After a simple calculation, we get from (\ref{3.14}) that if
\begin{align*}
\|(u_0,\phi_0)\|_{{\mathcal{H}^1}\times \mathcal{H}^{2}}\leq \sqrt{-\frac{c}{\ln(1-\varepsilon)}},
\end{align*}
there exists a set $\Lambda=\bigcap_{i\geq 1} \Omega_{2^{i}}$ such that $\mathbb{P}(\Lambda)\geq \varepsilon$ for all $\omega\in \Lambda$, and equation (\ref{Equ1.1}) has a unique global solution in $\widetilde{X}_T$. This completes the proof of Theorem \ref{the1.2}.

\section*{Acknowledgments}
Z.Qiu is supported by the CSC under grant No.201806160015.  H.Wang is supported by the National Natural Science Foundation of China (Grant No. 11901066), the
Natural Science Foundation of Chongqing (Grant No. cstc2019jcyj-msxmX0167) and Project No. 2019CDXYST0015 supported by the Fundamental Research Funds for the Central Universities.

\end{document}